\documentclass[12pt, reqno]{amsart}
\usepackage[margin=1in]{geometry}
\usepackage[english]{babel}
\usepackage[latin1]{inputenc}
\usepackage{amssymb}
\usepackage{amsmath}
\usepackage{arydshln}
\usepackage{booktabs}
\usepackage{graphicx}
\usepackage{epsfig}
\usepackage{enumitem}
\usepackage{hyperref}
\usepackage{mathrsfs}
\usepackage{mathtools}
\usepackage[new]{old-arrows}
\usepackage[all]{xy}
\usepackage[usenames, dvipsnames]{color}
\usepackage{url}

\usepackage{graphicx}
\usepackage{caption}
\usepackage{subcaption}
\usepackage[]{hyperref}

\newcommand{\E}{\mathcal{E}}
\newcommand{\C}{\mathcal{C}}
\newcommand{\R}{\mathbb{R}}
\newcommand{\N}{\mathbb{N}}

\newcommand{\Ss}{\mathcal{S}}
\newcommand{\T}{\mathcal{T}}
\newcommand{\D}{\mathcal{D}}
\newcommand{\M}{\mathcal{M}}

\newcommand{\Rem}{\operatorname{Rem}}

\newtheorem{theorem}{Theorem}
\newtheorem{lemma}[theorem]{Lemma}

\newtheorem{conj}[theorem]{Conjecture}

\newtheorem*{question*}{Question}

\theoremstyle{definition}

\theoremstyle{remark}
\newtheorem{rem}[theorem]{Remark}
\newtheorem*{remark}{Remark}

\makeatletter
\let\c@table\c@figure % for (1)
\let\ftype@table\ftype@figure % for (2)
\makeatother

\begin{document}

\title{Degeneration of the spectral gap with negative Robin parameter}
\author{Derek Kielty}
\email{dkielty2@illinois.edu}
\address{Department of Mathematics, University of Illinois, Urbana, IL, 61801, USA}

\maketitle  %Do not remove this line

\begin{abstract}
The spectral gap of the Neumann and Dirichlet Laplacians are each known to have a sharp positive lower bound among convex domains of a given diameter. Between these cases, for each positive value of the Robin parameter an analogous sharp lower bound on the spectral gap is conjectured. In this paper we show the extension of this conjecture to negative Robin parameters fails by proving that the spectral gap of double cone domains are exponentially small, for each fixed parameter value.
\end{abstract}

\section{\textbf{Introduction}}
The Robin Laplacian eigenvalue problem with parameter $\alpha$ on a bounded Lipschitz domain $\D$ is
\begin{equation*}
%\label{eq:RobinPDE}
\begin{cases}
      -\Delta u = \lambda u \quad &\text{on} ~ \D, \\
      \partial_{\nu}u + \alpha u = 0 \quad &\text{on}~ \partial \D.
\end{cases}
\end{equation*} 
It has a discrete spectrum of eigenvalues for each $\alpha \in (-\infty, \infty)$:
\[\lambda_1(\D) < \lambda_2(\D) \leq \lambda_3(\D) \leq \dots \to \infty.\]
Recall that when $\alpha = 0$ this is the Neumann eigenvalue problem. The limiting case $\alpha = \infty$ corresponds to Dirichlet boundary conditions. See \cite[Chapter 4]{Henrot} for foundational material on the Robin problem.

In this paper we prove that when $\alpha < 0$ the spectral gap 
\[(\lambda_2 - \lambda_1)(\D)\]
can be arbitrarily small among domains $\D$ of a given diameter. Investigation of this result for $\alpha < 0$ was motivated by the following conjecture of Andrews, Clutterbuck, and Hauer \cite[\S 10]{ACH} for $\alpha \geq 0$: 
\begin{conj}[Robin gap]
\label{conj:Gap}
If $\alpha \geq 0$ and $\D \subset \R^n$ is a bounded convex domain then
\[(\lambda_2 - \lambda_1)(\D) \geq (\lambda_2 - \lambda_1)(I),\]
where $I \subset \R$ is an open interval of length equal to the diameter of $\D$.
\end{conj}

The Robin gap conjecture is known to hold in the Neumann case ($\alpha = 0$) due to Payne and Weinberger \cite{PW}. A new proof was given by Andrews and Clutterbuck \cite[pp. 901]{AC}. In the Dirichlet case ($\alpha = \infty$) it is known, not only for the Laplacian, but also for Schr\"odinger operators with a convex potential, thanks to Andrews and Clutterbuck \cite[Corollary 1.4]{AC}.

\begin{figure}
    \centering
    \includegraphics{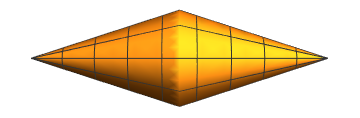}
    %\captionsetup{labelformat=empty}
    \caption{The \emph{double cone} domain $\D_{\theta}$ in 3-dimensions with $\theta = \pi/5$.}
    \label{fig:DblCn}
\end{figure}

To state our main result, define the family of convex \emph{double cone} domains of angle $\theta$ by
\[\D_{\theta} = \{(x,y) \in (-1,1) \times \R^{n-1} : |y| < \tan(\theta/2) (1 - |x|)\}, \quad \text{for each} ~ \theta \in (0,\pi),\]
as shown in Figure \ref{fig:DblCn}.

\begin{theorem}[Degeneration of the Robin gap]
\label{thm:main}
If $\alpha < 0$ then $(\lambda_2 - \lambda_1)(\D_{\theta}) \to 0$ as $\theta \to 0$. In fact, if $\alpha < 0$ then for each $\epsilon > 0$ there exists a constant $C > 0$ such that
\begin{equation}
\label{eq:GapInq}
    (\lambda_2 - \lambda_1)(\D_{\theta}) \leq C \exp\{-4(1 - \epsilon) |\alpha|/\theta\}, \quad \text{for all } \theta \text{ sufficiently small}.
\end{equation}
\end{theorem}

The scaling relation $\lambda_j(t\D;\alpha) = t^{-2} \lambda_j(\D;t\alpha)$ then shows that $(\lambda_2 - \lambda_1)(t \D_{\theta}) \to 0$ as $\theta \to 0$, for each $t > 0$. Since $t\D_{\theta}$ has diameter $2t$ when $\theta \in (0,\pi/2]$, Theorem \ref{thm:main} implies that among convex $\D$ of any given diameter, the spectral gap $(\lambda_2 - \lambda_1)(\D)$ can be made arbitrarily small when $\alpha < 0$. In particular, since $(\lambda_2 - \lambda_1)(I) > 0$ there exist convex domains $\D$ such that $(\lambda_2 - \lambda_1)(\D) < (\lambda_2 - \lambda_1)(I)$, and so the Robin gap conjecture fails to extend to $\alpha < 0$.

We believe the upper bound in Theorem \ref{thm:main} is sharp in the sense that the estimate (\ref{eq:GapInq}) fails when $\epsilon = 0$ as $\theta \to 0$. Although we do not prove this the proof of our upper bound on $\lambda_2(\D_{\theta})$ and Lemma \ref{lem:SF} suggests it is true.

\iffalse
The more ``correct" bound should be $\epsilon = 0$, but the polynomials factor $\theta^{n-1}$ in front.
\fi

The remainder of this paper is structured as follows: in Section 2 we discuss a heuristic explaining Theorem \ref{thm:main}, in Section 3 we present a brief overview of relevant literature, in Section 4 we prove a lemma about convergence of Schr\"odinger eigenvalues, and in Section 5 we prove Theorem \ref{thm:main}. Finally, in Section 6 we suggest a variant of Theorem \ref{thm:main} to a modification of the domains $\D_{\theta}$ and discuss the possibility of extending Conjecture \ref{conj:Gap} for $\alpha < 0$ to a restricted class of domains. 

\section{\textbf{A heuristic for small gaps in Theorem \ref{thm:main}: disjoint union of cones}}

The Robin eigenvalue problem with $\alpha < 0$ on the infinite cone of opening angle $\theta < \pi$
\begin{equation}
\label{eq:InfCone}
\C_{\theta} = \{(x,y) \in (0,\infty) \times \R^{n-1} : |y| < \tan(\theta/2) x\},
\end{equation}
has at least one negative eigenvalue and a corresponding ground state (see equation (\ref{eq:PhiDef})) that concentrates exponentially fast (in $L^2$-norm), as $\theta \to 0$. In the 2-dimensional case, it is known that \emph{all} the eigenfunctions concentrate in this fashion (see \cite[Theorem 5.1]{KP}). 

Because of this concentration we expect that the eigenfunctions of $\D_{\theta}$ concentrate at each vertex $(\pm 1, 0)$ like the ground state of $\C_{\theta}$ concentrates at the origin. Consequently, for small $\theta$ the Robin problem on $\D_{\theta}$ should be approximated by the Robin problem on two copies of $\C_{\theta}$. This approximation suggests $(\lambda_2 - \lambda_1)(\D_{\theta}) \to 0$ as $\theta \to 0$, because the first eigenvalue on $\C_{\theta} \sqcup \C_{\theta}$ has multiplicity two and hence the spectral gap equals zero.

Proving the gap tends to zero is a difficult task because $\lambda_1(\C_{\theta}) = -\alpha^2/\sin^2(\theta/2) \sim \theta^{-2}$ (see equation (\ref{eq:PhiDef})). Thus, Theorem \ref{thm:main} says that even though $\lambda_1(\D_{\theta})$ and $\lambda_2(\D_{\theta})$ have size $\theta^{-2}$, they are still exponentially close as $\theta \to 0$. Indeed, since $\lambda_1(\D_{\theta}) \leq \lambda_2(\D_{\theta})$, estimate (\ref{eq:LB}) in the proof of Theorem \ref{thm:main} actually shows that for each $\epsilon > 0$
\[\lambda_1(\D_{\theta}),\lambda_2(\D_{\theta}) = -\frac{\alpha^2}{\sin^2(\theta/2)} + O(\exp\{-4(1 - \epsilon) |\alpha|/\theta\}), \quad \text{ as } \theta \to 0.\]

\iffalse
Exponential concentration also occurs for all eigenfunctions on $\C_{\theta}$ in 2-dimensions as $\alpha \to -\infty$ for $\theta$ fixed, which is key to proving the asymptotics on polygons in \cite{Khalile}. 
\fi

\section{\textbf{Literature}}

For an overview of the Robin Laplacian and related shape optimization problems see the article by Bucur, Freitas, and Kennedy \cite[Chapter 4]{Henrot}. In the study of superconductors, Robin boundary conditions with negative parameter are also known for their connection to Ginzburg-Landau theory \cite[\S 1.4]{KS}.

Recently, Ashbaugh and the current author \cite{AK}, as well as Andrews, Clutterbuck, and Hauer \cite{ACH2}, proved sharp lower bounds on the Robin spectral gap for 1-dimensional Schr\"odinger operators with various classes of single-well potentials. In contrast with Theorem \ref{thm:main}, some of these lower bounds continue to hold for some $\alpha < 0$.

The Robin Laplacian on the infinite cone with small aperture will play a crucial role throughout this paper. Khalile and Pankrashkin \cite{KP} computed eigenvalue asymptotics for the small aperture limit of 2-dimensional infinite cones. Recently, this was also done in higher dimensions by Pankrashkin and Vogel \cite{PV}. This geometric limit is related to the $\alpha \to -\infty$ limit, for which Khalile \cite{Khalile} proved that the first $N = N(\D)$ Robin eigenvalues of a fixed curvilinear polygon $\D$ are determined (to leading order in $\alpha$) by the angles of the vertices. For smooth domains the second term in the $\alpha \to -\infty$ asymptotic is determined by the mean curvature; see work of Pankrashkin and Popoff \cite{PP}. In related work, B\"ogli, Kennedy, and Lang \cite{BKL} proved convergence of certain eigenvalues of the Robin Laplacian to those of the Dirichlet Laplacian for certain limits $\alpha \to \infty$, for complex $\alpha$.

\section{\textbf{Convergence of Schr\"odinger eigenvalues on expanding balls}}
%Robin Coulomb problem, expanding balls,...
In this section we prove the convergence and calculate the convergence rate as $R \to \infty$ of the first eigenvalue $E = E(R)$ of the Schr\"odinger eigenvalue problem
\begin{equation}
\label{eq:Schro}
\begin{cases}
      (-\Delta - \frac{\kappa}{r}) \varphi = E \varphi \quad &\text{on} ~ B(R), \\
      (\partial_r + \gamma) \varphi = 0 \quad &\text{on} ~ \partial B(R),
\end{cases}
\end{equation} 
where $B(R)$ is the ball of radius $R$ centered at the origin, $\kappa > 0$ is constant, and $\gamma \in (-\kappa/(n-1),\infty]$. When $\gamma = \infty$ we impose Dirichlet boundary conditions on $\partial B(R)$. Since the potential in (\ref{eq:Schro}) is radial, (\ref{eq:Schro}) reduces to an ODE eigenvalue problem that we analyze using special functions. In the proof of Theorem \ref{thm:main} in a later section we will get a lower bound on $\lambda_1(\D_{\theta})$ by proving a lower bound on $\lambda_1(\D_{\theta}) - \lambda_1(\C_{\theta})$ in terms of $E(R)$, where $R = \cot(\theta/2) \to \infty$ as $\theta \to 0$.

The aforementioned ODE eigenvalue problem gives rise to the \emph{confluent hypergeometric equation}
\begin{equation}
\label{eq:CHGE}
\rho \frac{d^2F}{d\rho^2} + (b - \rho)\frac{dF}{d\rho} - aF = 0,
\end{equation}
for certain $a \in \R$ and $b \notin \{0,-1,-2,\dots\}$, as we shall see in the proof of the next lemma. The solution $F = F(a,b;\rho)$ is a linear combination of the confluent hypergeometric functions of the first and second kind, $M(a,b;\rho)$ and $U(a,b;\rho)$, respectively. For more information on this ODE see Chapter 13 by Slater of the book edited by Abramowitz and Stegun \cite{AS} or Chapter 13 of the NIST Digital Library of Mathematical Functions \cite{NIST:DLMF}. 

In Lemma \ref{lem:SF} below, we show that the first eigenvalue of (\ref{eq:Schro}) is determined by a positive solution $z$ of the transcendental equation
\begin{equation}
\label{eq:trans}
    2\frac{a_0}{b_0} M(a_0 + 1,b_0 + 1; \rho_0) = (1 - 2\gamma z/\kappa)M(a_0,b_0; \rho_0),
\end{equation}
where 
\[m = (n-3)/2,~ a_0 = a_0(z) = m + 1-  z,~ b_0 = 2 m + 2,~ \rho_0 = \rho_0(z) = \kappa R/z.\] 
When $\gamma = \infty$ the transcendental equation is $M(a_0,b_0;\rho_0) = 0$. To see this case formally, divide by $\gamma$ in (\ref{eq:trans}) and let $\gamma \to \infty$.

\begin{lemma}
\label{lem:SF}
Assume that $\gamma \in (-\infty,\infty]$ and $\kappa > 0$ are fixed. 

(i) If $R < \infty$ then the radial eigenfunctions of (\ref{eq:Schro}) with $E < 0$ are multiples of 
\begin{equation}
\label{eq:SchroEF}
M(a_0,b_0;2\sqrt{|E|} ~ r)e^{-\sqrt{|E|} ~ r}
\end{equation}
and have eigenvalues $E = -\kappa^2/4 z^2$, where $z$ is a positive solution of (\ref{eq:trans}). 

(ii) If $R = \infty$ (i.e.\ the domain is $\R^n$) then the radial eigenfunctions (that decay as $r \to \infty$) of (\ref{eq:Schro}) are multiples of (\ref{eq:SchroEF}) with eigenvalues $E = -\kappa^2/4(j + (n - 3)/2)^2$ for $j \in \N$. Furthermore, $M(a_0,b_0;2\sqrt{|E|} r)$ is a polynomial in $r$ of degree $j-1$.

(iii) If $\gamma \in (-\kappa/(n-1),\infty]$ then for each $\epsilon > 0$, there is a solution $z_*$ of (\ref{eq:trans}) that corresponds to the first eigenvalue $E$ of (\ref{eq:Schro}) and satisfies $z_*(R) = (n-1)/2 + O(\exp\{-2(1-\epsilon) \kappa R/(n-1)\})$, as $R \to \infty$. Consequently,
\[E(R) = -\kappa^2/(n-1)^2 + O\bigg(\exp \bigg\{-(1 - \epsilon)\frac{2 \kappa }{n-1} R \bigg\}\bigg), \quad \text{as } R \to \infty.\]
\end{lemma}

\iffalse
\noindent The above lemma can likely extends to show that the spectrum of (\ref{eq:Schro}) converges to the spectrum of $-\Delta - \kappa/r$ on $\R^n$ as $R \to \infty$.
\fi

Part (iii) of the lemma shows that the first eigenvalue of (\ref{eq:Schro}) converges to $-\kappa^2/(n-1)^2$ as $R \to \infty$, which is the first eigenvalue of $-\Delta - \kappa/r$ on $\R^n$.

Computing the spectrum of the Schr\"odinger operator $-\Delta - \kappa/r$ on $B(R)$ and $\R^n$ is closely related to computing the spectrum of the hydrogen atom from quantum mechanics. In fact, on $\R^3$ this calculation is precisely the hydrogen atom calculation (see the book by Landau and Lifschitz \cite[Chapter V]{LL}, for example). For general $n \geq 3$, the ODE in (\ref{eq:Rad}) below is the radial part of the Schr\"odinger operator for the 3-D hydrogen atom with angular momentum $m$ (after changing variables). This physical interpretation breaks down when $n =2$ since in this case $m < 0$. The operator in (\ref{eq:Rad}) also appears in work of Khalile and Pankrashkin \cite{KP} as an effective operator for the Robin Laplacian with negative Robin parameter on $\C_{\theta}$ as $\theta \to 0$, when $n=2$. 

\begin{proof}[Proof of Lemma \ref{lem:SF}]

\textbf{Part (i):} 

Let $\varphi = \varphi(r,\xi)$ be an arbitrary radial eigenfunction of (\ref{eq:Schro}) with eigenvalue $E = E(R)$ written in $(r,\xi)$-spherical coordinates. Thus, $\varphi(r,\xi)$ is a constant times a function $u(r)$. In this case, $-\Delta - \kappa/r$ reduces to the radial part of the Laplacian plus the potential $-\kappa/r$ so that
\begin{equation*}
\begin{cases}
      -u'' - \frac{n-1}{r}u' - \frac{\kappa}{r} u = E u \quad &\text{on} ~ (0,R), \\
      u' + \gamma u = 0 \quad &\text{at} ~ r = R.
\end{cases}
\end{equation*}
In the above problem and the ones below we impose Dirichlet boundary conditions when $\gamma = \infty$. Let $m = (n-3)/2$ and $v(r) = r^{m+1} u(r)$ so that
\begin{equation}
\label{eq:Rad}
\begin{cases}
      -v'' + \frac{m(m + 1)}{r^2}v - \frac{\kappa}{r} v = E v \quad &\text{on} ~ (0,R), \\
      v' +(\gamma - \frac{m+1}{r})v = 0 \quad &\text{at} ~ r = R.
\end{cases}
\end{equation}

Since $E < 0$ we can define $ z = \kappa(2\sqrt{|E|})^{-1} > 0$, $\rho = \kappa r/  z$, and $w(\rho) = v( z \rho /\kappa)$. It follows that $w$ satisfies 
\begin{equation*}
\begin{cases}
      -w'' + \frac{m(m + 1)}{\rho^2}w - (\frac{ z}{\rho} - \frac{1}{4})w = 0 \quad &\text{on} ~ (0,\rho_0), \\
      w' + (\frac{ z \gamma}{\kappa} - \frac{m+1}{\rho})w = 0 \quad &\text{at} ~ \rho = \rho_0,
\end{cases}
\end{equation*}
where $\rho_0 = \rho_0(z) = \kappa R/ z$.

\iffalse
\[w'(\rho_0) + (\frac{ z \gamma}{\kappa} - \frac{m+1}{\rho_0})w(\rho_0) = 0\]
\[(w' + (\frac{ z \gamma}{\kappa} - \frac{m+1}{\rho_0})w)(\rho_0) = 0\]
\fi

Guessing a solution of the form $w(\rho) = \rho^{m + 1} e^{-\rho/2} F(\rho)$ we find that $F(\rho) = F(a_0,b_0;\rho)$ is a solution of (\ref{eq:CHGE}) on the interval $(0,\rho_0)$ with $a_0(z) = m +1 - z$, $b_0 = 2m + 2$, and boundary condition $2F'(a_0,b_0;\rho_0) = (1 -  2 z \gamma/\kappa)F(a_0,b_0;\rho_0)$. When $\gamma = \infty$ the boundary condition $F(a_0,b_0;\rho_0) = 0$ is immediate from the form of the ansatz. Thus, the original radial eigenfunctions $\varphi$ are multiples of $F(a_0,b_0;2\sqrt{|E|}r)e^{-\sqrt{|E|}r}$.

The general solution of
(\ref{eq:CHGE}) is a linear combination of the confluent hypergeometric functions of the first and second kind, $M(a_0,b_0;\rho)$and $U(a_0,b_0;\rho)$, respectively. It follows from formulas 13.4.21, 13.5.5, and 13.5.7 in \cite{AS} that
\[U'(a_0,b_0;\rho) = -a_0\frac{\Gamma(b_0)}{\Gamma(a_0+1)}\rho^{-b_0} + o(\rho^{-b_0}),\quad \text{as } \rho \to 0,\] 
so that $U'(a_0,b_0;\rho) \not\in L^2(B(R))$ since $b_0 = n-1$. Thus, $F(\rho)$ must have no $U(a_0,b_0;\rho)$ component to guarantee that $\varphi \in H^1(B(R))$. This shows that $F(\rho) = M(a_0,b_0;\rho)$.

Finally, using the boundary condition we computed for $F$ and the formula $M'(a,b;\rho) = \frac{a}{b}M(a+1,b+1;\rho)$ \cite[formula 13.4.8]{AS} we deduce that $z$ must satisfy the transcendental equation (\ref{eq:trans}). The definition of $z$ shows that a negative eigenvalue $E$ satisfies $E = -\kappa^2/4z^2$ for some solution $z$.

\iffalse
Conversely, each positive solution $z$ to the transcendental equation defines an eigenvalue $E = -\kappa^2/4 z^2$ corresponding to a radial eigenfunction of (\ref{eq:Schro}). Thus, the eigenvalues $E$ corresponding to radial eigenfunctions of (\ref{eq:Schro}) are in one-to-one correspondence with solutions $z$ of (\ref{eq:trans}), and part (i) of the lemma on $B(R)$ follows from this correspondence.
\fi

\bigskip

\noindent \textbf{Part (ii):} The same calculation shows that on $\R^n$ the eigenfunctions $\varphi$ are multiples of (\ref{eq:SchroEF}) with eigenvalue $-\kappa^2/4z^2$. Note that if $a$ and $b$ are fixed and $a \notin \{0,-1,-2,\dots\}$ then $M(a,b;\rho) \sim \rho^{a-b}e^\rho$ as $\rho \to \infty$ (see \cite[13.1.4]{AS}). If $a \in \{0,-1,-2,\dots\}$ then $M(a,b;\rho)$ is a polynomial of degree $-a$ (see (\ref{eq:ps})). Thus, we must have $a_0 \in \{0,-1,-2,\dots\}$ to ensure $\varphi$ tends to zero as $r \to \infty$. It follows from the definition of $a_0 = m+1 - z$ that $z = \kappa(2 \sqrt{|E|})^{-1}$ and $E = -\kappa^2/4(j + m)^2$ for $j \in \N$, where $m = (n-3)/2$.   

\bigskip

\noindent \textbf{Part (iii):} We split the the proof into two cases depending on the size of $|\gamma|$. Recall that $\kappa > 0$ and $m = (n-3)/2$. Let $0 < \epsilon < 1$.

\bigskip

\noindent \textbf{Case 1: $|\gamma| < \kappa/(2(m+1))$:} Put
\[z_-(R) = m+1 - e^{-(1 - \epsilon)\kappa R/(m+1)} \quad \text{and} \quad z_+ = m+1.\]
We show that there exists a $z_* = z_*(R) \in (z_-(R),z_+)$ that solves (\ref{eq:trans}) for all $R$ sufficiently large, and $z_*(R) \nearrow m + 1$ exponentially fast, as $R \to \infty$. We do this by applying the intermediate value theorem to the function 
\[\M(z) = 2\frac{a_0}{b_0} M(a_0 + 1,b_0 + 1; \rho_0) - (1 - 2\gamma z/\kappa)M(a_0,b_0; \rho_0),\]
where recall that $a_0 = a_0(z) = m + 1-  z$, $b_0 = 2 m + 2$, and $\rho_0 = \rho_0(z) = \kappa R/z$.
(Note that $\M(z) = 0$ if and only if $z$ satisfies the transcendental equation (\ref{eq:trans}).) The desired convergence for $E(R) = -\kappa^2/4z_*(R)^2$ as $R \to \infty$ then follows immediately. 

We show that $\M(z_+)< 0$ and that $\M(z) > 0$ for $z  \leq z_-(R)$ in order to apply the intermediate value theorem on the interval $[z_-,z_+]$ for each fixed $R$ sufficiently large. To prove these estimates on $\M$ we use the power series representation 
\begin{equation}
\label{eq:ps}
M(a,b;\rho) = \sum_{k = 0}^{\infty} \frac{a^{(k)}}{b^{(k)}} \frac{\rho^k}{k!}, \quad \text{for } a \in \R,~ b \notin \{0,-1,-2,\dots\},
\end{equation}
where $a^{(k)}$ is the rising factorial. That is, $a^{(0)} = 1$ and $a^{(k)} = a(a+1) \cdots (a + k - 1)$ for $k \geq 1$. Observe that $z \mapsto M(a_0(z),b;\rho_0(z))$ is continuous for each fixed $R$, since the series converges uniformly in $z$ on compact subsets of $(0,\infty)$ by the Weierstrass M-test. 

Note that $M(a_0,b_0;~\cdot~)$ is positive by (\ref{eq:ps}) when $0 < z < m+1$ because $a_0(z) > 0$. Thus, as soon as we prove that a solution $0 < z_* < m+1$ exists we know that it corresponds to the first eigenvalue $E$ of (\ref{eq:Schro}) since the ground state is the unique positive eigenfunction.

\bigskip

First we prove $\M(z_+) < 0$ for each $R > 0$. Indeed, since $a_0(m+1) = 0$ we know
\[\M(z_+) = \M(m+1) = -(1 - 2 \gamma (m+1)/\kappa)M(0,b_0;\rho_0) < 0\]
because $\gamma < \kappa/(2(m+1))$ and $M(0,b_0;~\cdot~)$ is identically 1. 

\bigskip

Now we prove $\M(z) > 0$ when $z \leq z_-(R)$ and $R$ is sufficiently large. First choose $\delta = \delta(\gamma) \in (0,1)$ large enough that $z \mapsto (\delta + 2 \gamma z/\kappa)$ is uniformly positive on $[0,m+1]$, which is possible since $\gamma > -\kappa/(2(m+1))$.  Next, observe that there is a $K = K(\gamma,m)$ independent of $z$ such that
\begin{equation}
\label{eq:FactorLB}
2\frac{k + m + 1 -  z}{k + 2m + 2} \geq 1 + \delta, \quad \text{for each } k \geq K.
\end{equation}
\noindent The $k^{\text{th}}$ coefficient of the power series of the first term of $\M$ satisfies the following identity, by considering the first and last factors in the rising factorials:
\[2 \frac{a_0}{b_0} \frac{(a_0 + 1)^{(k)}}{(b_0 + 1)^{(k)}} = 2\frac{m + 1 -  z}{2m + 2}\frac{(m + 2 - z)^{(k)}}{(2 m + 3)^{(k)}} = 2\frac{k + m + 1 - z}{k + 2m + 2} \frac{(m + 1 - z)^{(k)}}{(2 m + 2)^{(k)}},\]
for each $k \geq 0$. Applying the lower bound (\ref{eq:FactorLB}) we have
\begin{equation}
\label{eq:coef}
2 \frac{a_0}{b_0} \frac{(a_0 + 1)^{(k)}}{(b_0 + 1)^{(k)}} \geq (1 + \delta) \frac{a_0^{(k)}}{b_0^{(k)}}, \quad \text{for each } k \geq K.
\end{equation}

Dropping the first $K$ terms in the power series of $M(a_0 + 1,b_0 + 1;\rho_0)$, applying (\ref{eq:coef}), and adding and subtracting $1 + \delta$ times the first $K$ terms in the series for $M(a_0,b_0;\rho_0)$ shows
\begin{align}
2\frac{a_0}{b_0} M(a_0 + 1,b_0 + 1; \rho_0) \geq (1 + \delta) \sum_{k = K}^{\infty} \frac{a_0^{(k)}}{b_0^{(k)}} \frac{\rho_0^k}{k!} = (1 + \delta) M(a_0,b_0;\rho_0) - (1 + \delta) \sum_{k = 0}^{K - 1} \frac{a_0^{(k)}}{b_0^{(k)}} \frac{\rho_0^k}{k!}. \nonumber
\end{align}
Subtracting and adding $(2\gamma z/\kappa)M(a_0,b_0;\rho_0)$ to the right side shows
\begin{equation}
\label{eq:EE}
2\frac{a_0}{b_0} M(a_0 + 1,b_0 + 1; \rho_0) \geq (1 - 2\gamma z/\kappa) M(a_0,b_0;\rho_0)  + \E(\rho_0,z),
\end{equation}
where
\begin{equation*}
%\label{eq:LogPoly}
\E(\rho_0,z) = (\delta + 2\gamma z/\kappa) M(a_0,b_0;\rho_0) - (1 + \delta) \sum_{k = 0}^{K - 1} \frac{(m + 1 -  z)^{(k)}}{(2 m + 2)^{(k)}} \frac{\rho_0^k}{k!}.
\end{equation*}

To conclude that $\M(z) > 0$ from (\ref{eq:EE}) we will show that $\E(\rho_0,z) > 0$ for $z \leq z_-(R)$, when $R$ is sufficiently large. In what follows $C$ is a positive constant that is independent of both $k$ and $z$ and may change from line to line. Since $z \mapsto (\delta + 2 \gamma z/\kappa)$ is uniformly positive on $[0,m+1]$ by our choice of $\delta$ and the second term of $\E(\rho_0,z)$ satisfies
\begin{equation}
\label{eq:MExpBd}
(1 + \delta) \sum_{k = 0}^{K - 1} \frac{(m + 1 -  z)^{(k)}}{(2 m + 2)^{(k)}} \frac{\rho_0^k}{k!} \leq C(1 + \rho_0^K),
\end{equation}
to obtain $\E(\rho_0,z) > 0$ it suffices to show that \[M(a_0,b_0;\rho_0) \geq Ce^{(\epsilon/2)\rho_0}.\] 
(Note $\rho_0 = \kappa R/z \geq \kappa R/m \to \infty$, as $R \to \infty$ and so $e^{(\epsilon/2)\rho_0}$ dominates $\rho_0^K$.)

We estimate the factor $a_0^{(k)}/b_0^{(k)}$ from below using that $m + 1 -  z > 0$ and converting the rising factorials to ordinary factorials so that for $k \geq 1$,
\[\frac{a_0^{(k)}}{b_0^{(k)}} = \frac{(m + 1 -  z)^{(k)}}{(2 m + 2)^{(k)}} \geq (m + 1 -  z) \frac{(k-1)!}{(2 m + 2)^{(k)}} = (m + 1 -  z) \frac{(k-1)!(2m + 1)!}{(2 m + 2 + (k-1))!}.\]
Noting that $(2m + 1)!$ only depends on the dimension $n$ and making cancellations leads to the lower bound
\[\frac{a_0^{(k)}}{b_0^{(k)}} \geq (m + 1 -  z) \frac{C}{(k + 2 m + 1)^{2 m + 2}}, \quad \text{for } k \geq 0.\]

Using this lower bound, observing $(k + 2m + 1)^{-(2m + 2)} \geq C(1 - \epsilon/2)^k$, and using the Taylor series for $e^\rho$, we have
\begin{align}
M(a_0,b_0;\rho_0) \geq &C(m + 1 -  z) \sum_{k = 0}^{\infty} \frac{1}{(k + 2m + 1)^{2m+2}} \frac{\rho_0^k}{k!}  \nonumber \\
\geq &C(m + 1 -  z) \sum_{k = 0}^{\infty} \frac{((1 - \epsilon/2)\rho_0)^k}{k!}  \nonumber \\
= &C(m + 1 -  z)e^{(1 - \epsilon/2) \rho_0}. \nonumber 
\end{align}
Finally, using that $z < z_-(R) \leq m + 1 - e^{-(1- \epsilon) \rho_0}$ shows that 
\[(m + 1 -  z)e^{(1 - \epsilon/2) \rho_0} \geq e^{(\epsilon/2) \rho_0},\] 
and hence (\ref{eq:MExpBd}) holds.

\bigskip

Since $z \mapsto \M(a_0(z),b_0;\rho_0(z))$ is continuous the intermediate value theorem shows that there exists a solution $z_*$ to (\ref{eq:trans}) that satisfies $z_-(R) < z_* < z_+ = m + 1$. Using that $m + 1 = (n-1)/2$ we have that
\[\frac{n-1}{2} - \exp\bigg\{-(1 - \epsilon)\frac{2\kappa}{n-1} R\bigg\} \leq z_* \leq \frac{n-1}{2},\]
for all $R$ sufficiently large. This shows the first estimate in part (iii) of the lemma for $|\gamma| < \kappa/(2(m+1))$. Hence,
\begin{equation}
\label{eq:Easymp}
E(R) = -\frac{\kappa^2}{4z_*(R)^2} = -\frac{\kappa^2}{(n-1)^2} + O\bigg(\exp\bigg\{-(1 - \epsilon)\frac{2 \kappa }{n-1} R\bigg\}\bigg), \quad \text{as } R \to \infty.
\end{equation}
This completes the proof of part (iii) of the lemma for $|\gamma| < \kappa/(2(m+1))$.

\bigskip

\noindent \textbf{Case 2: $\gamma \in [\kappa/(2(m+1)),\infty]$:} In the remainder of the proof let $E(R,\gamma)$ denote the first eigenvalue of the Schr\"odinger problem (\ref{eq:Schro}). To prove the asymptotic for $E(R,\gamma)$ when $\gamma \in [\kappa/(2(m+1)),\infty]$ notice that $\gamma \mapsto E(R,\gamma)$ is increasing for $\gamma \in (-\infty,\infty]$ by the min-max principle. In particular, $E(R,0) \leq E(R,\gamma) \leq E(R,\infty)$ for each $\gamma > 0$. Thus, applying formula (\ref{eq:Easymp}) with $\gamma = 0$ gives the desired lower bound for $E(R,\gamma)$ and it suffices to prove the right side of (\ref{eq:Easymp}) is also an upper bound for the first Dirichlet eigenvalue $E(R,\infty)$. 

By part (ii) of the lemma, the ground state of the $-\Delta - \kappa/r$ on $\R^n$ is $\exp\{-\sqrt{|E(\R^n)|} ~ r\}$ with eigenvalue $E(\R^n) = -\kappa^2/(n-1)^2$. Let $\chi$ be a smooth radial cutoff function that is 1 for $r < 1 - \epsilon/2$, decreases to zero over $1 - \epsilon/2 \leq r \leq 1$, and is zero for $r > 1$. The function $f(r)$ defined by $\chi(r/R) \exp\{-\sqrt{|E(\R^n)|} ~ r\}$ restricted to the ball $B(R)$ is a valid trial function for the first eigenvalue $E(R,\infty)$ of (\ref{eq:Schro}) so that
\begin{align}
E(R,\infty) &\leq \frac{\int_{B(R)} f(-\Delta - \kappa/r)f \, dV}{\int_{B(R)} f^2 \, dV}\nonumber \\ 
&= \frac{E(\R^n)\int_{B((1 - \epsilon)R)} \exp\{-2\sqrt{|E(\R^n)|} ~ r\} \, dV  + \int_{B(R) \setminus B((1 - \epsilon)R)} f(-\Delta - \kappa/r)f \, dV}{\int_{B((1 - \epsilon)R)} \exp\{-2\sqrt{|E(\R^n)|} ~ r\} \, dV + \int_{B(R) \setminus B((1 - \epsilon)R)} f^2 \, dV}\nonumber \\
&= E(\R^n) + O\bigg(\exp\{-(1 - \epsilon)\frac{2 \kappa }{n-1} R\}\bigg), \quad \text{as } R \to \infty. \nonumber
\end{align}
The last equality follows from using the decay of $\exp\{-\sqrt{|E(\R^n)|} ~ r\}$ and that $\chi(r/R)$ has bounded $r$-derivatives. These facts show that the second terms in the numerator and denominator are the same size as the error term in (\ref{eq:Easymp}) and gives the desired upper bound since $E(\R^n) = -\kappa^2/(n-1)^2$. 

\bigskip

This completes the proof of the asymptotic for $E(R,\gamma)$ for all $\gamma \in (-\kappa/(2(m+1)),\infty]$. Note that the asymptotic for $z_*$ in part (iii) of the lemma follows from the formula $z_* = \kappa (2 \sqrt{|E(R,\gamma)|})^{-1}$.
\end{proof}

\section{\textbf{Proof of Theorem \ref{thm:main}}}
%\label{sec:Proofs}

We prove Theorem \ref{thm:main} by getting an upper bound on $\lambda_2(\D_{\theta})$ and a lower bound on $\lambda_1(\D_{\theta})$ in terms of the first eigenvalue of the infinite cone $\C_{\theta}$ (defined in (\ref{eq:InfCone})). Specifically, for each $\epsilon > 0$ we show that there exists a $C > 0$ such that
\begin{equation}
\label{eq:UB}
\lambda_2(\D_{\theta}) \leq \lambda_1(\C_{\theta}) + C \exp\{2(1 - \epsilon)\alpha/ \sin(\theta/2)\}, \quad \text{as } \theta \to 0,
\end{equation}
and
\begin{equation}
\label{eq:LB}
\lambda_1(\D_{\theta}) = \lambda_1(\C_{\theta}) + O(\exp\{2(1 - \epsilon)\alpha/ \sin(\theta/2)\}), \quad \text{as } \theta \to 0.
\end{equation}
Combining these bounds and using that $\sin(\theta/2) \leq \theta/2$ we conclude that there is a $C > 0$ such that $(\lambda_2 - \lambda_1)(\D_{\theta}) \leq C \exp\{-4(1 - \epsilon) |\alpha|/\theta\}$ for all $\theta$ sufficiently small. This concludes the proof of Theorem \ref{thm:main}.

To prove both inequality (\ref{eq:UB}) and (\ref{eq:LB}) we will need to use the ground state of the infinite cone $\C_{\theta}$ to construct trial functions. A direct calculation shows the $L^2$-normalized ground state of the Robin Laplacian on $\C_{\theta}$ is
\begin{equation}
\label{eq:PhiDef}
    \phi_{\theta}(x,y) = A_{\theta} e^{\alpha x/\sin(\theta/2)}, \quad \text{when } \alpha < 0 \text{ and } \theta < \pi,
\end{equation}
with eigenvalue 
\[\lambda_1(\C_{\theta}) = -\alpha^2/\sin^2(\theta/2),\]
where the constant in (\ref{eq:PhiDef}) satisfies
\begin{equation}
\label{eq:Atheta}
A_{\theta}^{-2} = 2^{-n} |\alpha|^{-n} \omega_{n-1} \Gamma(n) \tan^{n-1}(\theta/2) \sin^n(\theta/2).
\end{equation}
Here $\omega_{n-1}$ is the volume of the $(n-1)$-dimensional unit sphere and $\Gamma$ is the Gamma function. 

\bigskip

\subsection{Upper bound on $\lambda_2(\D_{\theta})$:}  In order to get an upper bound on $\lambda_2(\D_{\theta})$ we transplant two copies of $\phi_{\theta}$ to construct a trial function. Fix $0 < \epsilon < 1/2$ and let $\chi = \chi(x,y)$ denote a smooth nonnegative cutoff function on $\R^n$ that is independent of $y$, equals 1 for $x \leq 1 - \epsilon$, 0 for $x \geq 1$, and is less than or equal to 1 everywhere. Let $\psi_{\theta} = \chi \phi_{\theta}$ be a cut-off cone ground state and 
\[F(x,y) = (x+1,y) \quad \text{and} \quad G(x,y) = (1-x,y)\]
be rigid motions. Define the trial function $\widetilde \phi_{\theta} : \D_{\theta} \to \R$ by putting a cut-off cone ground state at each vertex of $\D_{\theta}$:
\begin{equation*}
%\label{eq:approxGS}
\widetilde \phi_{\theta} = \frac{(\psi_{\theta} \circ F) - (\psi_{\theta} \circ G)}{\sqrt{2}}.
\end{equation*}
Notice that $\widetilde \phi_{\theta}$ is odd with respect to $x$, by construction. 

Since $\D_{\theta}$ is reflection symmetric in $x$ each eigenfunction with simple eigenvalue is either even or odd in $x$. The ground state must be even since it is non-negative. Thus, the odd function $\widetilde \phi_{\theta}$ is orthogonal to the ground state, and so it is a valid trial function for
$\lambda_2(\D_{\theta})$. By the min-max principle,
\begin{equation}
\label{eq:L2ub}
\lambda_2(\D_{\theta}) \leq \frac{\int_{\D_{\theta}} |\nabla \widetilde \phi_{\theta}|^2 \, dV + \alpha \int_{\partial \D_{\theta}} \widetilde \phi_{\theta}^2 \, dS }{\int_{\D_{\theta}} \widetilde \phi_{\theta}^2 \, dV} = \frac{\int_{\T_{\theta}} |\nabla \psi_{\theta}|^2 \, dV + \alpha \int_{ \Sigma_{\theta}} \psi_{\theta}^2 \, dS }{\int_{\T_{\theta}} \psi_{\theta}^2 \, dV},
\end{equation}
where we define
\begin{equation*}
%\label{eq:TCDef}
\T_{\theta} = \C_{\theta} \cap \{x < 1\} \quad \text{and} \quad \Sigma_{\theta} = \partial \T_{\theta} \cap \{0 \leq x < 1\}
\end{equation*}
to be the truncated cone and the curved part of its boundary. 

A straightforward calculation (see below) will show there are remainder terms 
\[\Rem_i(\theta) = O(\exp\{2(1 - 2\epsilon)\alpha/ \sin(\theta/2)\}), \quad \text{as } \theta \to 0, \quad \text{for } i = 1,2,3,\] 
such that
\[\int_{\T_{\theta}} |\nabla \psi_{\theta}|^2 \, dV = \int_{\C_{\theta}} |\nabla \phi_{\theta}|^2 \, dV + \Rem_1(\theta), \quad \int_{ \Sigma_{\theta}} \psi_{\theta}^2 \, dS = \int_{\partial \C_{\theta}} \phi_{\theta}^2 \, dS + \Rem_2(\theta),\]
and
\[\int_{\T_{\theta}} \psi_{\theta}^2 \, dV = \int_{\C_{\theta}} \phi_{\theta}^2 \, dV + \Rem_3(\theta) = 1 + \Rem_3(\theta),\]
as $\theta \to 0$, where the final equality uses the $L^2$-normalization of $\phi_{\theta}$. Combining these estimates with the upper bound (\ref{eq:L2ub}) shows we have a remainder term $\Rem(\theta) = (\Rem_1(\theta) + \alpha \Rem_2(\theta))/(1 + \Rem_3(\theta)) = O(\exp\{2(1 - 2\epsilon)\alpha/ \sin(\theta/2)\})$ such that
\[\lambda_2(\D_{\theta}) \leq \frac{\int_{\C_{\theta}} |\nabla \phi_{\theta}|^2 \, dV + \alpha \int_{\partial \C_{\theta}} \phi_{\theta}^2 \, dS }{\int_{\T_{\theta}} \psi_{\theta}^2 \, dV} + \Rem(\theta) = \frac{\lambda_1(\C_{\theta})}{\int_{\T_{\theta}} \psi_{\theta}^2 \, dV} + \Rem(\theta) \leq \lambda_1(\C_{\theta}) + \Rem(\theta),\]
as $\theta \to 0$, since $\int_{\T_{\theta}} \psi_{\theta}^2 \, dV \leq 1$ and $\lambda_1(\C_{\theta}) < 0$. This proves (\ref{eq:UB}) since $\epsilon$ is arbitrary.

We estimate $\Rem_3(\theta)$; the other remainder terms are similar. The inequality $\int_{\T_{\theta}} \psi_{\theta}^2 \, dV \leq \int_{\C_{\theta}} \phi_{\theta}^2 \, dV$ follows from $\chi$ being less than 1. For an inequality in the opposite direction, since 
\[\int_{T_{\theta}} \psi_{\theta}^2 \, dV \geq \int_{\C_{\theta} \cap \{x < 1 - \epsilon\}} \phi_{\theta}^2 \, dV = \int_{\C_{\theta}} \phi_{\theta}^2 \, dV - \int_{\C_{\theta} \cap \{x \geq 1 -  \epsilon\}} \phi_{\theta}^2 \, dV,\]
it suffices to show the second term on the right is exponentially small. Since $\phi_{\theta}$ is constant in $y$ we have that
\begin{align*}
\int_{\C_{\theta} \cap \{x \geq 1 - \epsilon\}} \phi_{\theta}^2 \, dV &= A_{\theta}^2 \int_{1 - \epsilon}^{\infty} \exp\{2\alpha x/\sin(\theta/2)\} \text{Vol}_{n-1}(B(\tan(\theta/2)x)) \, dx\\
&= \Gamma(n)^{-1} \int_{2(1 - \epsilon)|\alpha|/\sin(\theta/2)}^{\infty} \exp\{-z\} z^{n-1} \, dz,
\end{align*}
where we have made the change of variables $z = 2|\alpha|x/\sin(\theta/2)$ and used the definition of $A_{\theta}$ in (\ref{eq:Atheta}). Since $\exp\{-z\} z^{n-1} \lesssim \exp\{-\frac{1 - 2\epsilon}{1 - \epsilon} z\}$ for all sufficiently large $z$, by integrating we have
\[\int_{\C_{\theta} \cap \{x \geq 1 - \epsilon\}} \phi_{\theta}^2 \, dV \lesssim \exp\{2(1 - 2\epsilon) \alpha/\sin(\theta/2)\}, \quad \text{as } \theta \to 0,\]
which shows that $\Rem_3(\theta) = O(\exp\{2(1 - 2\epsilon) \alpha/\sin(\theta/2)\})$.

\bigskip

\subsection{Lower bound on $\lambda_1(\D_{\theta})$:} We prove the lower bound by viewing $\lambda_1(\D_{\theta}) - \lambda_1(\C_{\theta})$ as the first eigenvalue of a problem on the truncated cone $\T_{\theta}$. Then we ``push out" the problem on $\T_{\theta}$ to a radial problem on a spherical sector $\Ss_{\theta}$. Finally, we transform the problem on $\Ss_{\theta}$ into the Schr\"odinger eigenvalue problem (\ref{eq:Schro}) whose solutions were analyzed in Lemma \ref{lem:SF}. We break up the proof into four steps.

\bigskip

\noindent \textbf{Step 1:} Let $u$ denote the $L^2$-normalized ground state of $\D_{\theta}$ restricted and translated to $\T_{\theta} = \C_{\theta} \cap \{x < 1\}$. Since the ground state of $\D_{\theta}$ is even in $x$, we know $u$ has Neumann boundary conditions on the flat part of the boundary defined by
\[\Gamma_{\theta} = \overline{\T_{\theta}} \cap \{x = 1\}\] 
and retains its Robin boundary conditions on the complement $\Sigma_{\theta} = \partial \T_{\theta} \cap \{0 \leq x < 1\}$. 

For notational convenience write $\phi$ instead of $\phi_{\theta}$ for the Robin ground state of $\C_{\theta}$. A direct calculation shows that the ratio $v = u/\phi$ is the first eigenfunction of the $\tau$-Laplacian operator $\Delta_{\tau}(\cdot) = \tau^{-1} \text{div}(\tau \nabla(\cdot))$ with weight $\tau = \phi^2$. More specifically, $v$ satisfies
\begin{equation}
\label{eq:fLapT}
\begin{cases}
-\Delta_{\tau} v = \mu_1 v & \text{on } \T_{\theta},\\
\partial_{\nu} v + (\alpha/\sin(\theta/2))v = 0 & \text{on } \Gamma_{\theta},\\
\partial_{\nu} v = 0 & \text{on }  \Sigma_{\theta},
\end{cases}
\end{equation}
where $\mu_1 = \mu_1(\T_{\theta}) = \lambda_1(\D_{\theta}) - \lambda_1(\C_{\theta})$ is the first eigenvalue of (\ref{eq:fLapT}) since $v$ is positive.

Since the weight $\phi^2$ is uniformly positive on $\overline{\T_{\theta}}$, the above eigenvalue problem has a discrete spectrum with Rayleigh quotient 
\[\frac{\int_{\T_{\theta}} |\nabla f| \phi^2 \, dx + (\alpha/\sin(\theta/2))\int_{\Gamma_{\theta}} f^2 \phi^2 \, dS}{\int_{\T_{\theta}} f^2 \phi^2 \, dx}, \quad \text{for } f \in H^1(\T_{\theta}) \setminus \{0\}.\]
Note that $\mu_1(\T_{\theta}) < 0$ by using the min-max principle with a constant trial function.
\bigskip

\noindent \textbf{Step 2:} Consider the spherical coordinates $(r,\xi)$, where $\xi = (\xi_1,\dots, \xi_{n-1})$ is the vector of angle coordinates and $\xi_1$ is the angle between the point $(x,y)$ and the $x$-axis. We will show $\mu_1(\Ss_{\theta}) \leq (1 +o(1))\mu_1(\T_{\theta})$, as $\theta \to 0$, where $\mu_1(\Ss_{\theta})$ is the first eigenvalue of a certain radial problem on the spherical sector 
\[\Ss_{\theta} = \{(r,\xi) : 0 < r < \cos(\theta/2)^{-1}, ~ |\xi_1| < \theta/2\}.\]
Note that this sector is obtained by ``pushing out" the flat part of the boundary of $\T_{\theta}$. To define this problem on $\Ss_{\theta}$ let the ``push in" diffeomorphism $P : \overline{\Ss_{\theta}} \to \overline{\T_{\theta}}$ be defined in spherical coordinates by
\[P(r,\xi) = (p r, \xi),\] 
where $p = p(\xi_1) = \cos(\theta/2)\sec(\xi_1)$. Also define the Rayleigh quotient as
\[\frac{\int_{\Ss_{\theta}} |\nabla f|^2 (\phi^2 \circ P) \, dV + (\beta/\sin(\theta/2))\int_{P^{-1}(\Gamma_{\theta})} f^2 (\phi^2 \circ P) \, dS}{\int_{\Ss_{\theta}} f^2 (\phi^2 \circ P) \, dV}, \quad \text{for } f \in H^1(\Ss_{\theta}) \setminus \{0\}.\]
The parameter $\beta = \beta(\theta)$ is equal to $(1 + o(1)) \alpha$ and will be defined precisely at the end of Step 2. The above Rayleigh quotient corresponds to the weighted $\sigma$-Laplacian eigenvalue problem on $\Ss_{\theta}$ with $\sigma = \phi^2 \circ P$, with Robin conditions on the spherical part of the boundary $P^{-1}(\Gamma_{\theta}) = \partial \Ss_{\theta} \cap \partial B(\cos(\theta/2)^{-1})$, and Neumann conditions on $\Sigma_{\theta}$.

We take $f = v \circ P$ as a trial function for $\mu_1(\Ss_{\theta})$ so that 
\begin{equation}
\label{eq:SEvalBd}
\mu_1(\Ss_{\theta}) \leq \frac{\int_{\Ss_{\theta}} |\nabla [v \circ P]|^2 (\phi^2 \circ P) \, dV + (\beta/\sin(\theta/2))\int_{P^{-1}(\Gamma_{\theta})} [v \circ P]^2 (\phi^2 \circ P) \, dS}{\int_{\Ss_{\theta}} [v \circ P]^2 (\phi^2 \circ P) \, dV}.
\end{equation}
Let $J_P$ and $J_{P^{-1}}$ be the Jacobian matrices of $P$ and $P^{-1}$, $I_{\Ss_{\theta}}$ and $I_{\T_{\theta}}$ be the identities on $\overline{\Ss_{\theta}}$ and $\overline{\T_{\theta}}$, and $I$ be the identity matrix. After writing $P$ and $P^{-1}$ in rectangular coordinates a calculation shows that
\begin{equation}
\label{eq:SigmaEst}
\lVert P - I_{\Ss_{\theta}} \rVert_{\infty}, ~\lVert J_P - I \rVert_{op,\infty}, ~\lVert P^{-1} - I_{\T_{\theta}} \rVert_{\infty}, ~\lVert J_{P^{-1}} - I \rVert_{op,\infty} \to 0, \quad \text{as } \theta \to 0.
\end{equation}
Here $\lVert \cdot \rVert_{\infty}$ is the supremum norm and $\lVert \cdot \rVert_{op,\infty}$ is the supremum norm of the (pointwise) operator norm of a matrix valued function. 

We calculate the denominator of (\ref{eq:SEvalBd}) by changing variables and using (\ref{eq:SigmaEst}) so that
\[\int_{\Ss_{\theta}} [v \circ P]^2 (\phi^2 \circ P) \, dV = \int_{\T_{\theta}} v^2 \phi^2 |\det(J_{P^{-1}})| \, dV = (1 + o(1))\int_{\T_{\theta}} v^2 \phi^2 \, dV, \quad \text{as } \theta \to 0.\]

Now we prove an upper bound on the gradient term in (\ref{eq:SEvalBd}): the chain rule gives
\[\int_{\Ss_{\theta}} |\nabla [v \circ P]|^2 (\phi^2 \circ P) \, dV \leq \lVert J_P \rVert_{\infty}^2 \int_{\Ss_{\theta}} |(\nabla v) \circ P|^2 (\phi^2 \circ P) \, dV \leq M(\theta) \int_{\T_{\theta}} |\nabla v|^2  \phi^2  \, dV,\]
where $M(\theta) = \lVert J_P \rVert_{\infty}^2 \lVert \det(J_{P^{-1}})\rVert_{\infty} = 1 + o(1)$, as $\theta \to 0$ by applying the estimates in (\ref{eq:SigmaEst}). 

Finally, to handle the boundary term we prove the lower bound
\begin{equation}
\label{eq:BdryTerm}
\int_{P^{-1}(\Gamma_{\theta})} [v \circ P]^2 (\phi^2 \circ P) \, dS \geq m(\theta) \int_{\Gamma_{\theta}} v^2 \phi^2 \, dS,
\end{equation}
where $m(\theta) = 1 + o(1)$ as $\theta \to 0$. By parametrizing the spherical cap $P^{-1}(\Gamma_{\theta})$ with the map $P^{-1}|_{\Gamma_{\theta}}$ we have
\begin{align*}
\int_{P^{-1}(\Gamma_{\theta})} [v \circ P]^2 (\phi^2 \circ P) \, dS = \int_{\Gamma_{\theta}} v^2 \phi^2 \sqrt{\det(G)}\, dy \geq m(\theta) \int_{\Gamma_{\theta}} v^2 \phi^2\, dy,
\end{align*}
where $G = G(y)$ is the Gramian matrix with entries $G_{ij} = \langle \partial_i P^{-1}|_{\Gamma_{\theta}}, \partial_j P^{-1}|_{\Gamma_{\theta}} \rangle$ with $P^{-1}|_{\Gamma_{\theta}}$ written in rectangular coordinates $y_i$ for $i = 1, \dots, n-1$ and $m(\theta) = \min\{\sqrt{\det(G)}\}$. The estimates in (\ref{eq:SigmaEst}) imply that $\partial_i P^{-1}|_{\Gamma_{\theta}} = e_i + o(1)$, where $e_i$ is the $i^{\text{th}}$ standard basis vector, so that $G_{ij} = \delta_{ij} + o(1)$ uniformly as $\theta \to 0$. Thus,  $m(\theta) = 1 + o(1)$ and (\ref{eq:BdryTerm}) holds.

Applying the estimates on each part of the Rayleigh quotient to (\ref{eq:SEvalBd}) shows that
\[\mu_1(\Ss_{\theta}) \leq (1 + o(1)) \bigg[M(\theta)\int_{\T_{\theta}} |\nabla v|^2 \phi^2\, dV + (\beta(\theta) m(\theta)/\sin(\theta/2))\int_{\Gamma_{\theta}} v^2 \phi^2 \, dS \bigg] \bigg/ \int_{\T_{\theta}} v^2 \phi^2 \, dV,\]
as $\theta \to 0$. Choosing $\beta(\theta) = M(\theta) m(\theta)^{-1} \alpha = (1 + o(1))\alpha$ and factoring out the $M(\theta)$ from the numerator shows that $\mu_1(\Ss_{\theta}) \leq (1 +o(1))\mu_1(\T_{\theta})$, as $\theta \to 0$, completing the proof of Step 2.

\bigskip

\noindent \textbf{Step 3:} We claim that $\mu_1(\Ss_{\theta})$ is the first eigenvalue of the problem
\begin{equation}
\label{eq:BallEv}
\begin{cases}
      -\Delta_{\sigma} w = \mu w \quad &\text{on} ~ B(\cos(\theta/2)^{-1}), \\
      \partial_r w + (\beta/\sin(\theta/2)) w = 0 \quad &\text{on} ~ \partial B(\cos(\theta/2)^{-1}),
\end{cases}
\end{equation} 
where $\sigma = \phi^2 \circ P$ and $\partial_r$ is the radial derivative (i.e.\ the normal derivative). Indeed, $\sigma = \exp\{2 \alpha \cot(\theta/2) r\}$ is radial, and so the first eigenfunction $w$ is also radial (by uniqueness), and therefore the restriction $w|_{\Ss_{\theta}}$ has Neumann boundary conditions on $ \Sigma_{\theta}$. The restriction is still an eigenfunction, and since it has the correct boundary conditions and is non-negative it must be a ground state of the problem on $\Ss_{\theta}$ and so $\mu_1(\Ss_{\theta})$ is the first eigenvalue of (\ref{eq:BallEv}). 

Now we transfer the $\theta$ dependence of $\sigma$ to the domain by rescaling the radial coordinate by $r \mapsto \cot(\theta/2)r$ and letting $R = \sin(\theta/2)^{-1}$ and $\tilde w(r) = w(\tan(\theta/2)r)$ so that
\begin{equation}
\label{eq:hLap}
\begin{cases}
      -\Delta_{\tilde \sigma} \tilde w = \nu_1 \tilde w \quad &\text{on} ~ B(R), \\
      \partial_r \tilde w + \tilde \beta \tilde w = 0 \quad &\text{on} ~ \partial B(R),
\end{cases}
\end{equation} 
where $\tilde \beta = \tilde \beta(\theta) = \cos(\theta/2)^{-1} \beta(\theta)$, $\tilde \sigma(r) = \sigma(\tan(\theta/2) r) = e^{2 \alpha r}$, and $\nu_1(R) = \tan^2(\theta/2) \mu_1(\Ss_{\theta})$.

\bigskip

\noindent \textbf{Step 4:} We make a final change of variables to eliminate the weight $\tilde \sigma$ and Robin parameter $\alpha$ so that (\ref{eq:hLap}) becomes the Schr\"odinger eigenvalue problem (\ref{eq:Schro}) with a shifted potential. A direct calculation shows that the radial function $\varphi(r) = \tilde w(r) /e^{-\alpha r}$ satisfies the Schr\"odinger eigenvalue problem
\begin{equation*}
\begin{cases}
      (-\Delta - \frac{\kappa}{r}) \varphi = (\nu_1 - \alpha^2) \varphi \quad &\text{on} ~ B(R), \vspace{1mm}\\
      \partial_r \varphi + (\tilde \beta - \alpha)\varphi = 0 \quad &\text{on} ~ \partial B(R),
\end{cases}
\end{equation*}
where $\kappa = (n-1)|\alpha|$. Since $\varphi \geq 0$ we know that $\nu_1 - \alpha^2$ is the first eigenvalue of this problem.

The definition of $\tilde \beta$ (at the end of Step 3) shows that $\tilde \beta -\alpha = o(1)$ as $R \to \infty$. Thus, $\gamma < \tilde \beta - \alpha$ for all $R$ sufficiently large, where $\gamma \in (-|\alpha|,0)$ is the fixed Robin parameter in the Schr\"odinger problem (\ref{eq:Schro}) with eigenvalue $E(R)$. Because the eigenvalues of Schr\"odinger operators are increasing functions of the Robin parameter we have that $E(R) \leq \nu_1(R) - \alpha^2$ for $R$ sufficiently large. Using that $R = \sin(\theta/2)^{-1}$
and combining this inequality with those from Steps 1, 2, and 3, there is a constant $C > 0$ such that for all small $\theta$,
\[C(R^2 - 1)(E(R) + \alpha^2) \leq C(R^2 - 1)\nu_1(R) = C\mu_1(\Ss_{\theta}) \leq \mu_1(\T_{\theta}) = \lambda_1(\D_{\theta}) - \lambda_1(\C_{\theta}) < 0.\]

Fix an $\epsilon > 0$. Using part (iii) of Lemma \ref{lem:SF}, we have that $E(R) + \alpha^2 = O(\exp\{-2(1 - \epsilon)|\alpha|R\})$ as $R \to \infty$. Hence
\[\lambda_1(\D_{\theta}) - \lambda_1(\C_{\theta}) = O(\exp\{-2(1 - \epsilon)|\alpha|/\sin(\theta/2)\}) \quad \text{as } \theta \to 0,\]
where this last estimate follows by using that $R = \sin(\theta/2)^{-1}$ and that the factor of $R^2 - 1$ multiplying $E(R) + \alpha^2$ can be absorbed into the error term since $\epsilon$ is arbitrary. Thus, (\ref{eq:LB}) is proven.

\begin{remark}
To prove the lower bound for $\lambda_1(\D_{\theta})$ we used the map $P^{-1}$ to ``push out" the truncated cone $\T_{\theta}$ to the spherical sector $\Ss_{\theta}$. This idea was used to compute bounds on eigenvalues of thin isosceles triangles by Freitas in \cite[\S 3]{Freitas} for Dirichlet boundary conditions and by Laugesen and Suideja in \cite[proof of Lemma 5.2]{LS} for Neumann boundary conditions. Here we adapt the technique to Robin boundary conditions.
\end{remark}

\section{\textbf{Questions}}

Theorem \ref{thm:main} shows that Conjecture \ref{conj:Gap} fails to extend to $\alpha < 0$ for general convex $\D$. In a positive direction Laugesen \cite[Theorem III.8]{L} recently showed that Conjecture \ref{conj:Gap} holds and is sharp for all $\alpha \in (-\infty,\infty]$ for rectangular boxes (i.e.\ products of intervals). This inspires the question: \emph{Is there a general class of convex domains such that the gap inequality in Conjecture \ref{conj:Gap} holds when $\alpha < 0$?} Since the proof of Theorem \ref{thm:main} appears to rely on the fact that the domain $\D_{\theta}$ has vertices with small opening angle, we might consider minimizing $(\lambda_2 - \lambda_1)(\cdot)$ over convex domains with a uniformly bounded Lipschitz constant. 

Surprisingly, it seems that the gap still fails to have a positive lower bound over this restricted class of domains. To see this we sketch an approximation argument. Define the \emph{truncated double cone} domains
\[\D_{\theta,\epsilon} = t_{\epsilon} \{(x,y) \in \D_{\theta} : |x| < 1 - \epsilon\}, \quad \text{for } \epsilon \in [0,1),\]
where the rescaling $t_{\epsilon}$ is such that $\D_{\theta,\epsilon}$ has diameter 2 for each $\epsilon$ sufficiently small. (Note that  $t_{\epsilon} = 1 + o(1)$, as $\epsilon \to 0$.) For $\theta$ fixed, $\{\D_{\theta,\epsilon}\}_{\epsilon \in [0,1)}$ has a Lipschitz constant that is uniform in $\epsilon \in [0,1)$ since the angle between the vertical faces at $x = \pm t_{\epsilon} (1-\epsilon)$ of $\partial \D_{\theta,\epsilon}$ and the remainder of the boundary does not depend on $\epsilon$. In addition, $\D_{\theta,\epsilon} \to \D_{\theta}$ in the Hausdorff distance and $|\partial \D_{\theta,\epsilon}| \to |\partial \D_{\theta}|$ as $\epsilon \to 0$ so we expect that
\[\lambda_j(\D_{\theta,\epsilon}) \to \lambda_j(\D_{\theta}), \quad \text{as } \epsilon \to 0,\]
for each $j \in \N$; see \cite[Remark 4.33]{Henrot} for a discussion of Robin continuity results. 

Assuming the above limit, we can use Theorem \ref{thm:main} to find a ``diagonal sequence" $\epsilon(\theta)$ converging to zero such that
\[(\lambda_2 - \lambda_1)(\D_{\theta,\epsilon(\theta)}) \to 0, \quad \text{as } \theta \to 0,\]
and $\D_{\theta,\epsilon(\theta)}$ has diameter 2 for each $\theta$. Observe that the truncated domains $\{\D_{\theta,\epsilon(\theta)}\}_{\theta \in (0,\pi/2)}$ also have uniformly bounded Lipschitz constants. This shows that $\lambda_2 - \lambda_1$ can be made arbitrarily small among convex domains of a given diameter and with a uniformly bounded Lipschitz constant. This example demonstrates the difficulty of determining a general class of domains where Conjecture \ref{conj:Gap} holds when $\alpha < 0$, even with a weaker lower bound.

\section*{Acknowledgements}
The author would like to thank Richard Laugesen for his encouragement and many helpful conversations on this project, as well as Mark Ashbaugh for some enlightening conversations about special functions. The author gratefully acknowledges support from the University of Illinois Campus Research Board award RB19045 (to Richard Laugesen).

\bibliographystyle{plain}
\bibliography{refs}

\end{document}